\documentclass[reqno,10pt]{amsart}
\usepackage[left=1in,right=1in,top=.95in,bottom=.95in]{geometry}
\usepackage{amsmath,amssymb,amsthm,mathrsfs,mathtools}
\usepackage{esint}
\usepackage{microtype}
\usepackage{xcolor}
\usepackage[colorlinks=true,linkcolor=blue,citecolor=red,urlcolor=blue]{hyperref}
\usepackage{enumitem}
\numberwithin{equation}{section}
\allowdisplaybreaks[2]
\numberwithin{equation}{section}

\theoremstyle{plain}
\newtheorem{theorem}{Theorem}[section]
\newtheorem{proposition}[theorem]{Proposition}
\newtheorem{lemma}[theorem]{Lemma}
\newtheorem{corollary}[theorem]{Corollary}

\theoremstyle{remark}
\newtheorem{remark}[theorem]{Remark}

\newcommand{\R}{\mathbb{R}}
\newcommand{\Sn}{\mathbb{S}^{n}}
\newcommand{\dd}{\,\mathrm{d}}
\newcommand{\cN}{c_n}

\title[The constant $Q/R$-curvature equation]
{Liouville Theorem for the Constant
$Q/R$-Curvature Equation}

\author[H. Lu]{Han Lu}

\address{School of Mathematics and Statistics, Henan University, Kaifeng 475004, China}
\email{hlu@henu.edu.cn}



\begin{document}

\begin{abstract}
We study positive entire solutions of the flat constant $Q/R$-curvature
equation in dimension $n\geq5$.  We prove that the weak superharmonicity
condition $-\Delta u\geq0$ already forces every nonconstant solution to be admissible.   With this automatic admissibility, we classify all such
solutions as positive constants or standard bubbles.  We also show
that the weak superharmonicity condition is
structurally necessary.  Once it is removed, there is a one-parameter family
of positive radial entire solutions near every positive constant.  These
solutions converge to a positive constant at infinity, while
both $-\Delta u$ and $-\Delta(u^{(n-2)/(n-4)})$ change sign infinitely many times.
\end{abstract}

\maketitle

\section{Introduction}

The Yamabe problem is the basic second-order model for conformal curvature
equations, while the Paneitz operator and Branson's $Q$-curvature provide a
fourth-order analogue.  Ge, Wang, and Wei recently introduced a Yamabe-type
problem for the quotient of $Q$-curvature by scalar curvature and singled out
the classification of positive entire solutions of its flat equation as an
open Euclidean rigidity problem \cite{GeWangWei26}.  We study this flat
model, proving classification under weak superharmonicity. We also show that,
without a sign condition, positive radial scattering solutions emerge from
the constant solutions.

We consider the normalization for which the quotient agrees with that of the
unit round sphere.  Throughout the paper, $n\geq5$ and
\begin{equation*}
 p=\frac{n-2}{n-4},\qquad
 q=\frac{n}{n-4}=2p-1,
\end{equation*}
and
\begin{equation*}
 \cN=\frac{(n+2)(n-4)}4.
\end{equation*}
Our flat equation is
\begin{equation}
 \Delta^2u=\cN u^{p-1}\bigl[-\Delta(u^p)\bigr]
 \qquad\hbox{in }\R^n,\qquad u>0.
 \label{eq:main}
\end{equation}

To see its geometric meaning, let
\begin{equation*}
 g=u^{\frac4{n-4}}|dx|^2=(u^p)^{\frac4{n-2}}|dx|^2
\end{equation*}
and set
\begin{equation*}
 w=-\Delta(u^p).
\end{equation*}
The conformal transformation laws for scalar curvature and $Q$-curvature give
\begin{align*}
 R_g=\frac{4(n-1)}{n-2}
 u^{-\frac{n+2}{n-4}}w,\qquad
 Q_g=\frac2{n-4}u^{-\frac{n+4}{n-4}}\Delta^2u
 =\frac{2\cN}{n-4}u^{-\frac{n+2}{n-4}}w.
\end{align*}
Thus, wherever $R_g\ne0$, \eqref{eq:main} is equivalent to
\begin{equation}
 \frac{Q_g}{R_g}=\frac{n^2-4}{8(n-1)}.
 \label{eq:Q-equals-kappa-R}
\end{equation}
For convenience, throughout this paper we refer to $w\geq0$ as \emph{weak admissibility}, and to
$w>0$ as \emph{strict admissibility}.  The strict condition is exactly
\(R_g>0\).

The standard solutions are the conformal factors obtained from the round
sphere by stereographic projection and M\"obius transformations:
\begin{equation}
 U_{\xi,\lambda}(x)
 =\left(\frac{2\lambda}{\lambda^2+|x-\xi|^2}\right)^{\frac{n-4}{2}},
 \qquad \xi\in\R^n,\quad \lambda>0.
 \label{eq:bubble}
\end{equation}
Every positive constant also solves the cross-multiplied equation
\eqref{eq:main}, although both $R_g$ and $Q_g$ then vanish and the literal
quotient in \eqref{eq:Q-equals-kappa-R} is undefined.

For this $Q/R$ Yamabe-type problem, Ge--Wang--Wei proved a
sharp quotient Sobolev inequality on the round sphere, an Obata-type rigidity
theorem on compact Einstein manifolds, and existence results on closed
manifolds under natural curvature-positivity assumptions \cite{GeWangWei26}.
Related recent refinements of the Obata--V\'etois rigidity method may be found
in \cite{Vetois24,Case24}.  Li--Wang--Wei subsequently analyzed blow-up
for the quotient equation and constructed noncompact families in dimensions
$n\geq25$ \cite{LiWangWei26}. 

Our first result gives the classification under a condition weaker than
nonnegative scalar curvature.

\begin{theorem}
\label{thm:weak-superharmonic}
Let $u\in C^4(\R^n)$ be a positive solution of \eqref{eq:main} and assume
\begin{equation*}
 -\Delta u\geq0\qquad\hbox{in }\R^n.
\end{equation*}
Then either $u$ is a positive constant or $u=U_{\xi,\lambda}$ for some
$\xi\in\R^n$ and $\lambda>0$.  More precisely, every nonconstant solution
satisfies
\begin{equation}
 -\Delta u\geq K u^q+(p-1)\frac{|\nabla u|^2}{u},
 \qquad K=\frac{\cN p}{q},
 \label{eq:weak-gradient-upgrade-stated}
\end{equation}
and consequently
\begin{equation*}
 -\Delta(u^p)\geq pK u^{p+q-1}>0.
\end{equation*}
\end{theorem}

Since $w\geq0$ implies
$-\Delta u\geq (p-1)|\nabla u|^2/u\geq0$, Theorem~\ref{thm:weak-superharmonic}
contains the following classification theorem which assumes nonnegative scalar curvature.  We nevertheless state this result separately because its proof supplies the geometric core of
the argument.

\begin{theorem}
\label{thm:main}
Let $u\in C^4(\R^n)$ be positive and satisfy \eqref{eq:main} and
$w=-\Delta(u^p)\geq0$.  Then either $u$ is a positive constant or
$u=U_{\xi,\lambda}$ for some $\xi\in\R^n$ and $\lambda>0$. 
\end{theorem}

The weak-superharmonicity hypothesis in
Theorem~\ref{thm:weak-superharmonic} cannot be dropped.
The following result constructs a family of positive entire solutions near the
constant background.  It also explains why the equation $Q_g=CR_g$ and
the quotient equation $Q_g/R_g=C$ must be distinguished when the scalar
curvature is allowed to vanish or change sign.

\begin{theorem}
\label{thm:nonadm-scattering}
For every $n\geq5$ there exists $\varepsilon_n>0$ such that, for every
$0<|s|<\varepsilon_n$, equation \eqref{eq:main} has a positive smooth radial
solution $u_s$ satisfying
\begin{equation*}
 u_s(0)=1,\qquad u_s'(0)=0,\qquad u_s''(0)=-s.
\end{equation*}
There exists $\ell_s>0$ such that $u_s(r)\to\ell_s$ as $r\to\infty$.
Moreover, if
\begin{equation*}
 k_s^2=\cN p\,\ell_s^{2(p-1)},
\end{equation*}
then there are constants $A_s,B_s$, not both zero, and a function $L_s$ such
that
\begin{align*}
 L_s(r)-\ell_s&=O(r^{3-n}),\\
 u_s(r)&=L_s(r)+r^{-\frac{n-1}{2}}
 \{A_s\cos(k_sr)+B_s\sin(k_sr)+o(1)\}.
\end{align*}
Both $-\Delta u_s$ and $-\Delta(u_s^p)$ change sign infinitely many times.
In particular, $u_s$ is neither a standard bubble nor weakly superharmonic.
\end{theorem}
\begin{remark}
    By the natural scaling $u(x)\mapsto c u(c^{p-1}x)$, the same construction is
available near every positive constant $c$.
\end{remark}
Theorems~\ref{thm:weak-superharmonic} and
\ref{thm:main} provide the corresponding Euclidean rigidity results, while
Theorem~\ref{thm:nonadm-scattering} describes what can occur without
superharmonicity.

We next place these results in the context of the classical and recent
literature.  The second-order prototype is the Yamabe problem, whose solution
was completed through the work of Yamabe, Trudinger, Aubin, and Schoen
\cite{Yamabe60,Trudinger68,Aubin76,Schoen84}.  On Euclidean space, the
moving-plane method of Gidas--Ni--Nirenberg \cite{GidasNiNirenberg79}
underlies the classification theorem of Caffarelli--Gidas--Spruck
\cite{CaffarelliGidasSpruck89}, which identifies the positive entire solutions
of the critical equation as the standard conformal bubbles.  This relation
between compact conformal rigidity and an entire Liouville theorem is the
second-order model for the present problem.

The corresponding fourth-order geometry is governed by the Paneitz operator
\cite{Paneitz08} and Branson's $Q$-curvature \cite{Branson85}.  Its flat
constant-$Q$ equation is the critical biharmonic equation, for which
Euclidean classification results were obtained by Lin \cite{Lin98} and, in
the higher-order conformally invariant setting, by Wei and Xu
\cite{WeiXu99}.  On closed manifolds, the constant-$Q$ problem has led to an
extensive existence theory; see, for example, Chang--Yang
\cite{ChangYang95} and Djadli--Malchiodi \cite{DjadliMalchiodi08}.

In contrast with the Yamabe equation, a general fourth-order equation has no
direct maximum principle.  Positivity properties of the Paneitz operator and
its Green function therefore play a central role; relevant results include
Gursky--Malchiodi \cite{GurskyMalchiodi15}, Gursky--Hang--Lin
\cite{GurskyHangLin16}, and Hang--Yang \cite{HangYang16}; a recent refinement,
including the remaining five-dimensional positivity question, is due to Li
\cite{Li26Paneitz}.

Curvature quotients form a related but analytically distinct class.  Ge and
Wang studied the second-order conformal quotient equations associated with
$\sigma_2/\sigma_1$ \cite{GeWang07,GeWang13}.  In such problems, positivity of
the denominator and the underlying admissible-cone condition are part of the
elliptic structure.

\medskip

We now outline the proofs.  The proof of Theorem~\ref{thm:main} begins by
setting \(z=-\Delta u\).  Under \(w\geq0\), a global Riesz decomposition gives
exact normalized Newton-potential representations for \(u\), \(u^p\), and
\(z\).  A nonlinear inequality for Newton potentials then gives the coercive
estimate
\begin{equation*}
 z\geq Ku^q.
\end{equation*}
Consequently \(T=z-Ku^q>0\), and after a Kelvin transformation the
fourth-order equation can be rewritten as the positive cooperative system
\begin{align*}
 -\Delta U&=V+AU^q,\\
 -\Delta V&=qA U^{q-1}(V+AU^q),
 \qquad A>0,
 \quad x\in\R^n\setminus\{0\}.
\end{align*}
We prove a punctured integral moving-plane theorem for this system.  The
argument follows the broad
integral moving-plane philosophy of Chen--Li--Ou \cite{ChenLiOu05}, but it must
accommodate two different decay rates, a possible B\^ocher atom, and a negative
power when $n\geq9$.  Since the equation and the normal-potential identities
are translation invariant, the resulting dichotomy may be applied with any
prescribed point as the inversion center: either the Kelvin puncture is
removable, or the original solution is radial about that center.

To close the classification, we apply this dichotomy at two distinct centers.
If either puncture is removable, the corresponding Kelvin transform is
bounded there and hence
\begin{equation*}
 u(x)=O(|x|^{4-n})
 \qquad (|x|\to\infty).
\end{equation*}
If neither puncture is removable, then \(u\) is radial about two distinct
points and must therefore be constant, contrary to the nonconstant reduction.
Thus the removable branch occurs at one of the two centers.  Stereographic
compactification and the Obata-type rigidity theorem of Ge--Wang--Wei
\cite{GeWangWei26} then identify the solution with one of the bubbles
\eqref{eq:bubble}.  The classical Obata theorem \cite{Obata71} provides the
underlying model for this final rigidity argument.

The mechanism behind Theorem~\ref{thm:weak-superharmonic} is the
 recovery of
admissibility.  Without assuming a sign for $w$, the quantity
\begin{equation*}
 T=z-Ku^q
\end{equation*}
satisfies the exact identity
\begin{equation*}
 -\Delta T=\cN p(p-1)u^{q-2}|\nabla u|^2\geq0.
\end{equation*}
Starting only from $z\geq0$, whole-space potential theory and a subharmonic
argument first yield $T>0$.  For $v=u+\varepsilon$, we introduce
$w_b=\Delta u+b|\nabla u|^2/v+Ku^q$ and bootstrap the inequality $w_b\leq0$
from $b=0$ to $b=p-1$, using the weighted quantities $v^{-b}w_b$.  At each
step, a direct calculation gives
$\Delta(v^{-b}w_b)\geq0$ on the positive set of $w_b$.  Kato's inequality
makes the positive part subharmonic, while a Newton-potential decomposition
gives a subvolume growth bound for $\int_{B_R}|\nabla u|$; together these
eliminate the positive part of $w_b$ globally.  The coefficient iteration yields
\eqref{eq:weak-gradient-upgrade-stated} and hence automatic admissibility.

Finally, the solutions in Theorem~\ref{thm:nonadm-scattering} arise from an
exact reduction to a radial Helmholtz oscillator plus a monotone slow
variable.  A bootstrap gives global positivity and convergence to a
constant, while comparison with the regular Bessel mode proves that the
scattering amplitude does not vanish.  The resulting oscillation forces both
$-\Delta u$ and $-\Delta(u^p)$ to change sign.

The paper is organized as follows.  Sections~2 and 3 establish normality and
the nonlinear Newton-potential inequality in the admissible class.  Section~4
derives the Kelvin-transformed cooperative system, and Section~5 proves the
punctured moving-plane theorem.  Section~6 establishes the compactification
result for fast-decaying solutions and completes the proof of
Theorem~\ref{thm:main} by the two-center argument.  Section~7 proves the global
$P$-function upgrade and Theorem~\ref{thm:weak-superharmonic}.  Section~8
constructs the non-admissible scattering family and proves
Theorem~\ref{thm:nonadm-scattering}.

\section{Potential theory and global normality}

Set
\begin{equation*}
 z=-\Delta u.
\end{equation*}
Expanding $w$ gives
\begin{equation*}
 w=p u^{p-1}z-p(p-1)u^{p-2}|\nabla u|^2
 =p u^{p-1}\left(z-(p-1)\frac{|\nabla u|^2}{u}\right).
\end{equation*}
Hence weak admissibility implies
\begin{equation*}
 z\geq (p-1)\frac{|\nabla u|^2}{u}\geq0.
\end{equation*}
The fourth-order equation is equivalently
\begin{equation}
 -\Delta z=\cN u^{p-1}w\geq0.
 \label{eq:z-equation}
\end{equation}
If $w\equiv0$, then the positive harmonic function $u^p$ is constant, and so
is $u$.  Henceforth, until the final proof of Theorem~\ref{thm:main}, we impose
the standing assumptions
\begin{equation*}
    w\geq0,\qquad w\not\equiv0.
\end{equation*}
Thus we work throughout the nonconstant admissible branch.  In particular,
$z\not\equiv0$, and the strong minimum principle applied to
\eqref{eq:z-equation} gives $z>0$.  Hence $u,u^p,z$ are positive
superharmonic functions on $\R^n$.

Let $I_2$ denote the normalized Newton potential
\begin{equation*}
 I_2f(x)=\frac{1}{(n-2)|\mathbb S^{n-1}|}
 \int_{\R^n}|x-y|^{2-n}f(y)\dd y,
 \qquad -\Delta I_2f=f,
\end{equation*}
whenever the nonnegative integral is finite, and let $I_4=I_2\circ I_2$.

We first establish the global Riesz representations for the three
superharmonic quantities introduced above.
\begin{lemma}
\label{lem:riesz-decomposition}
There is a constant $\ell\geq0$ such that
\begin{align}
 u&=\ell+I_2z, \label{eq:u-decomp}\\
 u^p&=\ell^p+I_2w, \label{eq:up-decomp}\\
 z&=\inf_{\R^n}z+\cN I_2(u^{p-1}w). \label{eq:z-decomp}
\end{align}
\end{lemma}

\begin{proof}
By the global Riesz decomposition theorem,
\begin{equation*}
 s=I_2(-\Delta s)+h,
\end{equation*}
where \(h\geq0\) is entire harmonic and \(I_2(-\Delta s)\) is the
Newton potential of the Riesz measure of \(s\).  Hence \(h\) is a
nonnegative constant by Liouville's theorem.

For later use,  we identify this constant and prove the corresponding
spherical-mean decay.  Let \(P=I_2\mu\) be a nonnegative Newton potential
which is finite at the origin.  Newton's shell formula gives
\begin{equation}
 \fint_{\partial B_R}P
 =\frac{1}{(n-2)|\mathbb S^{n-1}|}
 \left[
 R^{2-n}\mu(B_R)
 +\int_{\R^n\setminus B_R}|y|^{2-n}\,\dd\mu(y)
 \right].
 \label{eq:Newton-shell-mean}
\end{equation}
Since \(P(0)<\infty\), the weight \(|y|^{2-n}\) is
\(\mu\)-integrable.  Hence dominated convergence in \eqref{eq:Newton-shell-mean} gives
\begin{equation*}
 \fint_{\partial B_R}P\longrightarrow0
 \qquad\text{as }R\to\infty.
\end{equation*}

Since \(P\geq0\), it follows that \(\inf_{\R^n}P=0\).  Consequently,
in the Riesz decomposition above, $h=\inf_{\R^n}s$.
The same argument applies after translation to any fixed center.

Applying this to \(u\) and \(z\) gives
\eqref{eq:u-decomp} and \eqref{eq:z-decomp}.  Applying it to \(u^p\)
and using $\inf_{\R^n}u^p =\ell^p$ gives \eqref{eq:up-decomp}.
\end{proof}

We next show that the harmonic constants in the preceding Riesz
decompositions vanish.
\begin{lemma}
\label{lem:normality}
In fact $\ell=0$ and $\inf_{\R^n}z=0$.  Consequently,
\begin{equation}
 \boxed{
 u=I_2z=\cN I_4(u^{p-1}w),\qquad
 u^p=I_2w,\qquad
 z=\cN I_2(u^{p-1}w).}
 \label{eq:normal-system}
\end{equation}
\end{lemma}

\begin{proof}
If $\inf_{\R^n}z>0$, then $z\geq\inf_{\R^n}z>0$.  Inserting this into
\eqref{eq:u-decomp} would require
the Newton potential of a positive constant on all of $\R^n$, which is infinite
at every point.  Thus $\inf_{\R^n}z=0$.

Suppose next that $\ell>0$.  By \eqref{eq:z-decomp} and
\eqref{eq:up-decomp},
\begin{align*}
 z
 &=\cN I_2(u^{p-1}w)
 \geq \cN\ell^{p-1} I_2w
 =\cN\ell^{p-1}(u^p-\ell^p)\\
 &\geq \cN p\ell^{2(p-1)}(u-\ell).
\end{align*}
Here convexity of $s\mapsto s^p$ was used in the second line.  With
$h=u-\ell>0$, this becomes
\begin{equation*}
 -\Delta h\geq \kappa h,
 \qquad \kappa=\cN p\ell^{2(p-1)}>0.
\end{equation*}
Thus by Barta's inequality, $\lambda_1(B_R)\geq\kappa$.  Since
$\lambda_1(B_R)=R^{-2}\lambda_1(B_1)\to0$, this is impossible.  Hence
$\ell=0$, and \eqref{eq:normal-system} follows.
\end{proof}

\section{A nonlinear Newton-potential inequality}

The next lemma is the key pointwise estimate.  It eliminates the earlier
``minimum escaping to infinity'' obstruction for the ratio
$z/u^q$.

\begin{lemma}
\label{lem:potential-chain}
Let $G\geq0$ be locally integrable, suppose $P=I_2G<\infty$ pointwise, and let
$\beta>1$.  Then
\begin{equation}
 P^\beta\leq \beta I_2(P^{\beta-1}G)
 \qquad\hbox{on }\R^n.
 \label{eq:potential-chain}
\end{equation}
No finiteness assumption on $\int_{\R^n}G$ is required.
\end{lemma}

\begin{proof}
First take
\begin{equation*}
 G_{R,M}=\min\{G,M\}\mathbf1_{B_R},
 \qquad P_{R,M}=I_2G_{R,M}.
\end{equation*}
If \(G_{R,M}\equiv0\), there is nothing to prove.  Otherwise the strict
positivity of the Newton kernel gives \(P_{R,M}>0\) everywhere.  Hence,
after a standard smooth approximation if necessary, the chain rule below
is legitimate also when \(1<\beta<2\). Define
\begin{equation*}
 \Psi_{R,M}=\beta I_2\bigl(P_{R,M}^{\beta-1}G_{R,M}\bigr)-P_{R,M}^{\beta}.
\end{equation*}
In distributions,
\begin{equation*}
 -\Delta\Psi_{R,M}=\beta(\beta-1)P_{R,M}^{\beta-2} |\nabla P_{R,M}|^2
 \geq0.
\end{equation*}
For bounded compactly supported $G_{R,M}$, one has
$P_{R,M}(x)=O(|x|^{2-n})$.  Therefore the first term in $\Psi_{R,M}$ is
$O(|x|^{2-n})$, while
$P_{R,M}^\beta=O(|x|^{-\beta(n-2)})$.  In particular,
$\liminf_{|x|\to\infty}\Psi_{R,M}(x)\geq0$.  Applying the minimum principle on
large balls and sending the radius to infinity gives $\Psi_{R,M}\geq0$.

Letting first $M\to\infty$ and then $R\to\infty$, all nonnegative quantities
increase monotonically: $P_{R,M}\uparrow P$ and
$P_{R,M}^{\beta-1}G_{R,M}\uparrow P^{\beta-1}G$.  Monotone convergence yields
\eqref{eq:potential-chain}, even when the intermediate integrals are infinite.
In the application below the right-hand side is finite by
\eqref{eq:normal-system}.
\end{proof}

We now apply Lemma~3.1 to obtain the following estimate.
\begin{corollary}
\label{cor:z-lower}
Let
\begin{equation*}
 K=\frac{\cN p}{q}.
\end{equation*}
Then every solution under consideration satisfies
\begin{equation}
 z\geq Ku^q.
 \label{eq:z-lower}
\end{equation}
Moreover, if
\begin{equation*}
 T=z-Ku^q,
\end{equation*}
then
\begin{equation}
 T>0,\qquad
 -\Delta T=\cN p(p-1)u^{q-2}|\nabla u|^2\geq0.
 \label{eq:T-identity}
\end{equation}
\end{corollary}

\begin{proof}
Apply Lemma~\ref{lem:potential-chain} to
\begin{equation*}
 P=u^p=I_2w,
 \qquad
 \beta=\frac{q}{p}=\frac{n}{n-2}.
\end{equation*}
It follows that
\begin{equation*}
 u^q\leq \frac{q}{p} I_2(u^{p-1}w)=\frac{q}{p\cN}z=\frac{1}{K}z.
\end{equation*}
This is exactly \eqref{eq:z-lower}.

For the identity, expand
\begin{align*}
 -\Delta z
 &=\cN p u^{q-1}z-\cN p(p-1)u^{q-2}|\nabla u|^2,\\
 -\Delta(Ku^q)
 &=Kq u^{q-1}z-Kq(q-1)u^{q-2}|\nabla u|^2.
\end{align*}
Using $Kq=\cN p$, the $z$ terms cancel and give
\eqref{eq:T-identity}.  We already know $T\geq0$.  If $T$ vanished at an
interior point, the strong minimum principle would imply $T\equiv0$; then
\eqref{eq:T-identity} would force $\nabla u\equiv0$, contradicting the
nonconstant reduction made above.
Thus $T>0$.
\end{proof}

\section{Kelvin transform and a positive cooperative system}

The main obstacle to applying moving planes directly to \(u\) is the lack of
a priori control at infinity.  Fix any prescribed point as the inversion
center.  After translating coordinates and relabeling the translated solution,
we may take that point to be the origin and define
\begin{equation*}
 U(x)=|x|^{4-n}u\left(\frac{x}{|x|^2}\right),
 \qquad x\in\R^n\setminus\{0\}.
\end{equation*}
Write $\iota(x)=x/|x|^2$ and define
\begin{equation*}
 W(x)=|x|^{-n-2}w(\iota(x)).
\end{equation*}

The Newton-potential identities are exactly covariant under
this inversion.
\begin{lemma}
\label{lem:Kelvin-potentials}
On $\R^n\setminus\{0\}$,
\begin{equation}
 U^p=I_2W,
 \qquad
 U=\cN I_4(U^{p-1}W),
 \qquad
 Z:=-\Delta U=\cN I_2(U^{p-1}W)>0.
 \label{eq:Kelvin-normal}
\end{equation}
In particular, $Z$ has no unknown harmonic remainder.
\end{lemma}

\begin{proof}
For $\alpha\in\{2,4\}$, the Riesz kernel satisfies the Kelvin covariance
identity
\begin{equation}
 I_\alpha\left(|\cdot|^{-n-\alpha}H\circ\iota\right)(x)
 =|x|^{\alpha-n}(I_\alpha H)(\iota(x)).
 \label{eq:Riesz-Kelvin}
\end{equation}
It follows directly by inversion in the defining nonnegative integral, hence
Tonelli's theorem suffices and no finite-mass assumption is needed.

Apply \eqref{eq:Riesz-Kelvin} first to $u^p=I_2w$.  Since
$U^p=|x|^{2-n}u^p\circ\iota$, this yields $U^p=I_2W$.
Next apply it to $u=\cN I_4(u^{p-1}w)$.  Because
\begin{equation*}
 U^{p-1}W=|x|^{-n-4}(u^{p-1}w)\circ\iota,
\end{equation*}
we obtain $U=\cN I_4(U^{p-1}W)$.  Applying $-\Delta$ gives the last identity in
\eqref{eq:Kelvin-normal}.
For completeness, these transformed densities are locally integrable at the
puncture.  Indeed, finiteness of each of the two positive potentials at a fixed
point $x_0\ne0$, together with a positive lower bound for the corresponding
kernel on a small ball about the origin, gives
$W,U^{p-1}W\in L^1_{\mathrm{loc}}$, respectively.  Thus all distributional
differentiations above and below are legitimate across compact subsets of the
punctured space.
\end{proof}

Lemma~\ref{lem:potential-chain}, applied to
$G=W$ through the monotone double truncations
\begin{equation*}
 \min\{W,M\}\mathbf1_{\{\varepsilon<|x|<R\}},
\end{equation*}
and then with
$M,R\uparrow\infty$ and $\varepsilon\downarrow0$, gives
\begin{equation*}
 Z\geq KU^q.
\end{equation*}
Define
\begin{equation*}
 A=\frac K2=\frac{\cN p}{2q},
 \qquad
 V=Z-AU^q.
\end{equation*}
Using the identities for $U$ and $Z$ in \eqref{eq:Kelvin-normal}, together
with $I_4=I_2\circ I_2$, we obtain the exact formula
\begin{equation}
 U=I_2Z=I_2(V+AU^q).
 \label{eq:Kelvin-U-I2Z}
\end{equation}

With this choice of \(V\), the fourth-order equation reduces to a positive
cooperative second-order system.
\begin{lemma}
\label{lem:cooperative-system}
The pair $(U,V)$ is positive and satisfies
\begin{equation}
 \boxed{
 \begin{aligned}
 -\Delta U&=V+AU^q,\\
 -\Delta V&=qA U^{q-1}(V+AU^q)
 \end{aligned}}
 \qquad\hbox{in }\R^n\setminus\{0\}.
 \label{eq:cooperative-system}
\end{equation}
Moreover $V\geq AU^q>0$.
\end{lemma}

\begin{proof}
The lower bound follows from $Z\geq KU^q=2AU^q$.  To compute the second
equation, use the expansions in the proof of
Corollary~\ref{cor:z-lower}.  Since $Aq=\cN p/2$ and $q-1=2(p-1)$, the two
gradient-square terms cancel exactly:
\begin{align*}
 -\Delta V
 &=-\Delta Z-A[-\Delta(U^q)]\\
 &=\frac{\cN p}{2}U^{q-1}Z
 =qA U^{q-1}(V+AU^q).
\end{align*}
\end{proof}

The regularity of $u$ at the original inversion center gives the
differentiable far-field expansion
\begin{equation*}
 U(x)=u(0)r^{4-n}+r^{3-n}\nabla u(0)\cdot\theta
      +O_2(r^{2-n}),
 \qquad r=|x|,\quad\theta=x/r,
\end{equation*}
where $O_2(r^\gamma)$ denotes a remainder whose derivatives of order
$j\leq2$ are $O(r^{\gamma-j})$.  Hence
\begin{align}
 U(x)&=u(0)|x|^{4-n}\bigl(1+O(|x|^{-1})\bigr),
 \label{eq:U-tail}\\
 V(x)&=2(n-4)u(0)|x|^{2-n}\bigl(1+O(|x|^{-1})\bigr).
 \label{eq:V-tail}
\end{align}
Indeed, the displayed $O_2$ expansion may be differentiated twice, and
\begin{align*}
 AU^q&=O(r^{-n}),\\
 -\Delta\bigl(u(0)r^{4-n}\bigr)&=2(n-4)u(0)r^{2-n}.
\end{align*}
We first establish a standard local regularity fact that will be used to
remove the Kelvin puncture.
\begin{lemma}
\label{lem:distributional-removal}
Let $v\geq0$, $v\in L^1_{\mathrm{loc}}(B_1)$, and let
$a,f\in L^\infty(B_1)$.  Assume that
\begin{equation*}
 \int_{B_1}v(-\Delta\varphi)\,\dd x
 =\int_{B_1}(av+f)\varphi\,\dd x
 \qquad\text{for every }\varphi\in C_c^\infty(B_1).
\end{equation*}
Then
\begin{equation*}
 v\in L^\infty_{\mathrm{loc}}(B_1)
 \cap W^{2,r}_{\mathrm{loc}}(B_1)
 \qquad\text{for every }1<r<\infty.
\end{equation*}
\end{lemma}

\begin{proof}
Fix \(B_\rho \Subset B_{2\rho} \Subset B_1\), and choose
\(\eta\in C_c^\infty(B_{2\rho})\) with \(\eta\equiv 1\) on \(B_\rho\).
Since \(av+f\in L^1_{\mathrm{loc}}\), the local Newton representation
\begin{equation*}
v=I_2\bigl(\eta(av+f)\bigr)+h
\qquad\text{in }B_\rho,
\end{equation*}
with \(h\) harmonic, and the weak endpoint estimate for \(I_2\) give
\begin{equation*}
v\in L^s_{\mathrm{loc}},
\qquad
1<s<\frac{n}{n-2}.
\end{equation*}
Interior \(W^{2,s}\) estimates and Sobolev embedding then improve the
integrability exponent. Iterating finitely many times until $s>n/2$ yields $v\in L^\infty_{\mathrm{loc}}$. Hence \(av+f\in L^r_{\mathrm{loc}}\) for every finite \(r\), and a final
interior \(W^{2,r}\) estimate gives $v\in W^{2,r}_{\mathrm{loc}}(B_1)$ for every $1<r<\infty$.
\end{proof}
We now apply this to the Kelvin pair, and obtain the following result.
\begin{lemma}
\label{lem:Bocher-pair}
There exists $d\geq0$ such that
\begin{equation}
 V=d\Phi+I_2\left[qA U^{q-1}(V+AU^q)\right],
 \qquad
 \Phi(x)=\frac{1}{(n-2)|\mathbb S^{n-1}|}|x|^{2-n}.
 \label{eq:V-Bocher}
\end{equation}
If $U$ is bounded near $0$, then $d=0$ and the pair $(U,V)$ extends smoothly
and positively across $0$.
\end{lemma}

\begin{proof}
Set
\begin{equation*}
 g=qA U^{q-1}(V+AU^q).
\end{equation*}
By Riesz decomposition and B\^ocher's theorem,
\begin{equation*}
    V=d\Phi+I_2g+H,
\end{equation*}
where $H$ is an entire harmonic function.  As in the proof of
Lemma~\ref{lem:riesz-decomposition}, $d\Phi+I_2g\leq V$, so $H\geq0$.
Hence $H\equiv c\geq0$.  The decay \eqref{eq:V-tail} forces $c=0$, proving
\eqref{eq:V-Bocher}.

Suppose $U$ is bounded near zero.  If $d>0$, then the nonnegative
representation \eqref{eq:V-Bocher} gives $V\geq d\Phi$, and the exact identity
$U=I_2(V+AU^q)$ implies
\begin{equation*}
 U(x)\gtrsim I_2(\Phi\mathbf1_{B_\rho})(x)
 \gtrsim |x|^{4-n}
\end{equation*}
near zero, contradicting boundedness.  Hence $d=0$, so $V$ is
$L^1_{\mathrm{loc}}$ and we have the identity
\begin{equation*}
 \int_{B_\rho}V(-\Delta\varphi)\,\dd x
 =\int_{B_\rho}\bigl(qA U^{q-1}V+qA^2U^{2q-1}\bigr)\varphi\,\dd x
 \qquad(\varphi\in C_c^\infty(B_\rho))
\end{equation*}
for a sufficiently small $\rho>0$.  Since $U$ is bounded there, 
Lemma~\ref{lem:distributional-removal} therefore gives
$V\in L^\infty_{\mathrm{loc}}\cap W^{2,r}_{\mathrm{loc}}$ for every finite
$r$.  Since $U$ is bounded, $V+AU^q\in L^r_{\mathrm{loc}}$ for every finite
$r$; hence the equation for $U$ gives $U\in W^{2,r}_{\mathrm{loc}}$.
Thus $U,V\in C^{1,\alpha}_{\mathrm{loc}}$ for some $\alpha\in(0,1)$.
The extended equations and the strong minimum principle first give
$U,V>0$ across $0$.  The nonlinearities are consequently smooth near the
puncture, and standard local elliptic
bootstrapping then gives smoothness across \(0\).
\end{proof}

\section{Moving planes in the punctured space}

This section gives a self-contained moving-plane argument adapted to the
mixed decay rates $U\sim |x|^{4-n}$ and $V\sim |x|^{2-n}$.  We shall use the
two nonlinearities
\begin{equation*}
 H_1(s,t)=t+As^q,
 \qquad
 H_2(s,t)=qA s^{q-1}(t+As^q).
\end{equation*}
Thus $U=I_2H_1(U,V)$ exactly, whereas
$V=d\Phi+I_2H_2(U,V)$ by Lemma~\ref{lem:Bocher-pair}.

For $\lambda<0$, set
\begin{equation*}
 \Sigma_\lambda=\{x_1<\lambda\},
 \qquad
 x^\lambda=(2\lambda-x_1,x'),
 \qquad
 P_\lambda=2\lambda e_1.
\end{equation*}
Write $U_\lambda(x)=U(x^\lambda)$ and
$V_\lambda(x)=V(x^\lambda)$.  Notice that
$|x^\lambda|<|x|$ for $x\in\Sigma_\lambda$, and that $P_\lambda$ is the
reflection of the puncture.

The next two lemmas provide the common inputs for both the initial and the
continuation stages: the exact reflected identities and their linearization
on the bad sets.

\begin{lemma}
\label{lem:reflected-identities}
For $x\in\Sigma_\lambda\setminus\{P_\lambda\}$ define
\begin{equation*}
 \mathcal K_\lambda(x,y)=\Phi(x-y)-\Phi(x^\lambda-y),
 \qquad y\in\Sigma_\lambda.
\end{equation*}
Then $\mathcal K_\lambda(x,y)>0$, and
\begin{align}
 U(x)-U_\lambda(x)
 &=\int_{\Sigma_\lambda}\mathcal K_\lambda(x,y)
   \bigl[H_1(U,V)(y)-H_1(U_\lambda,V_\lambda)(y)\bigr]\dd y,
 \label{eq:exact-reflection-U}\\
 V(x)-V_\lambda(x)
 &=d\bigl[\Phi(x)-\Phi(x^\lambda)\bigr] \notag\\
 &\quad+\int_{\Sigma_\lambda}\mathcal K_\lambda(x,y)
   \bigl[H_2(U,V)(y)-H_2(U_\lambda,V_\lambda)(y)\bigr]\dd y.
 \label{eq:exact-reflection-V}
\end{align}
Moreover, $\Phi(x)-\Phi(x^\lambda)<0$.
\end{lemma}

\begin{proof}
Split each Newton-potential integral over $\Sigma_\lambda$ and its reflected
half-space, and make the change of variables $y\mapsto y^\lambda$ in the
latter.  This gives \eqref{eq:exact-reflection-U} and
\eqref{eq:exact-reflection-V}.  The sign assertions follow from
$|x-y|<|x^\lambda-y|$ for $x,y\in\Sigma_\lambda$ and
$|x^\lambda|<|x|$.
\end{proof}

Define the bad sets and positive parts
\begin{equation*}
 \begin{gathered}
 \Omega^U_\lambda=\{x\in\Sigma_\lambda:U(x)>U_\lambda(x)\},
 \qquad
 \Omega^V_\lambda=\{x\in\Sigma_\lambda:V(x)>V_\lambda(x)\},\\
 D_\lambda=(U-U_\lambda)_+,
 \qquad
 E_\lambda=(V-V_\lambda)_+.
 \end{gathered}
\end{equation*}

\begin{lemma}
\label{lem:bad-linearization}
Extend all functions below by zero outside their indicated bad sets.  There
are measurable coefficients $c_{ij}\geq0$ such that
\begin{align*}
 D_\lambda&\leq I_2(E_\lambda+c_{11}D_\lambda),
\\
 E_\lambda&\leq I_2(c_{21}D_\lambda+c_{22}E_\lambda),
\end{align*}
where, on $\Omega^U_\lambda$,
\begin{align}
 c_{11}&=Aq\,\xi^{q-1},
 \label{eq:c11}\\
 c_{21}&=qA(q-1)V\,\eta_1^{q-2}
       +qA^2(2q-1)\eta_2^{2q-2},
 \label{eq:c21}
\end{align}
and, on $\Omega^V_\lambda$,
\begin{equation}
 c_{22}=qA U_\lambda^{q-1}.
 \label{eq:c22}
\end{equation}
Here $\xi,\eta_1,\eta_2$ lie between $U$ and $U_\lambda$.
\end{lemma}

\begin{proof}
Take the positive part of \eqref{eq:exact-reflection-U} and use
$0<\mathcal K_\lambda(x,y)\leq\Phi(x-y)$.  The mean-value theorem applied to
$As^q$ gives \eqref{eq:c11}.  For $H_2$, use the rectangular decomposition
\begin{align*}
 H_2(U,V)-H_2(U_\lambda,V_\lambda)
 &=[H_2(U,V)-H_2(U_\lambda,V)]\\
 &\quad+[H_2(U_\lambda,V)-H_2(U_\lambda,V_\lambda)].
\end{align*}
The first bracket can be positive only on $\Omega^U_\lambda$ and gives
\eqref{eq:c21}; the second can be positive only on $\Omega^V_\lambda$ and
gives \eqref{eq:c22}.  The point-mass term is discarded because it has the
favorable sign found in Lemma~\ref{lem:reflected-identities}.
\end{proof}

The preceding inequalities are the common pointwise input for the entire
moving-plane argument.  We now state the symmetry result and organize its
proof around the smallness of the three coefficient norms.

\begin{proposition}
\label{prop:moving-planes}
Let $(U,V)$ be a positive solution of \eqref{eq:cooperative-system} satisfying
\eqref{eq:U-tail}, \eqref{eq:V-tail}, the exact potential identity for $U$ in
\eqref{eq:Kelvin-U-I2Z}, and the representation \eqref{eq:V-Bocher}.  Then
either $0$ is removable for $(U,V)$, or both $U$ and $V$ are radial about $0$.
\end{proposition}

\begin{proof}
If $U$ is bounded near $0$, Lemma~\ref{lem:Bocher-pair} gives removability.
We may therefore consider the nonremovable case.

\medskip
\noindent\emph{Step 1. A contraction criterion on the bad sets.}
We first isolate the analytic criterion used both to start the plane and to
continue it past a putative limiting position.  Suppose
$D_\lambda\in L^{2q}$ and $E_\lambda\in L^2$.  The
Hardy--Littlewood--Sobolev mappings \cite{Lieb83}
\begin{align*}
 I_2&:L^2\longrightarrow L^{2q},\\
 I_2&:L^{2n/(n+4)}\longrightarrow L^2
\end{align*}
are valid precisely because $n>4$.  Applying them to the inequalities in
Lemma~\ref{lem:bad-linearization}, followed by H\"older with
\begin{equation*}
 \frac12=\frac2n+\frac1{2q},
 \qquad
 \frac{n+4}{2n}=\frac4n+\frac1{2q},
\end{equation*}
gives
\begin{align*}
 \|D_\lambda\|_{L^{2q}}
 &\leq C\|E_\lambda\|_{L^2}
 +C\|c_{11}\|_{L^{n/2}(\Omega^U_\lambda)}
       \|D_\lambda\|_{L^{2q}},
\\
 \|E_\lambda\|_{L^2}
 &\leq C\|c_{21}\|_{L^{n/4}(\Omega^U_\lambda)}
       \|D_\lambda\|_{L^{2q}}
 +C\|c_{22}\|_{L^{n/2}(\Omega^V_\lambda)}
       \|E_\lambda\|_{L^2}.
\end{align*}
After absorbing the two diagonal terms, the second inequality bounds
$\|E_\lambda\|_2$ by a small multiple of $\|D_\lambda\|_{2q}$.  Substitution
into the first inequality gives the following criterion: there exists
$\varepsilon_0=\varepsilon_0(n)>0$ such that
\begin{equation}
 \|c_{11}\|_{L^{n/2}(\Omega^U_\lambda)}
 +\|c_{22}\|_{L^{n/2}(\Omega^V_\lambda)}
 +\|c_{21}\|_{L^{n/4}(\Omega^U_\lambda)}
 <\varepsilon_0
 \label{eq:HLS-smallness}
\end{equation}
forces
\begin{equation*}
 D_\lambda=E_\lambda=0.
\end{equation*}
Thus it remains only to make the three coefficient norms in
\eqref{eq:HLS-smallness} sufficiently small on the relevant bad sets.

\medskip
\noindent\emph{Step 2. Starting the plane.}
We verify the criterion of Step~1 when $\lambda$ is sufficiently negative.
Since $\Sigma_\lambda\subset\R^n\setminus B_{|\lambda|}$, it is enough to
establish coefficient bounds on the far-field parts of the bad sets,
uniformly in $\lambda$.

\emph{Claim.}
There exist \(R_0,C>0\), independent of \(\lambda<0\), such that on $(\Omega_\lambda^U\cup\Omega_\lambda^V)\setminus B_{R_0}$, for $r=|x|$,
\begin{equation*}
c_{11}+c_{22}\le Cr^{-4},
\qquad
c_{21}\le C\bigl(r^{-6}+r^{-8}\bigr).
\end{equation*}
Consequently,
\begin{equation}
 \begin{split}
 \sup_{\lambda<0}\bigl(
 &\|c_{11}\|_{L^{n/2}(\Omega^U_\lambda\setminus B_R)}
 +\|c_{22}\|_{L^{n/2}(\Omega^V_\lambda\setminus B_R)}\\
 &+\|c_{21}\|_{L^{n/4}(\Omega^U_\lambda\setminus B_R)}
 \bigr)\longrightarrow0
 \end{split}
 \label{eq:tail-norm-small}
\end{equation}
as $R\to\infty$.

To prove the claim, first assume $x\in\Omega^U_\lambda$. Put \(\rho=|x^\lambda|\le r\) and fix a large
threshold \(R_1\). We first prove that there exists $R_0>0$ independent of $\lambda$, such that for $r>R_0$, we must have $\rho>R_1$. Indeed, suppose the contrary, then the exact positive
potential representation
\begin{equation*}
U=I_2H_1(U,V)
\end{equation*}
gives a uniform positive lower bound $U(x^\lambda)\geq c(R_1)$, whereas
\eqref{eq:U-tail} gives \(U(x)\to0\) as
\(r\to\infty\). Hence, after increasing \(R_0\) if necessary,
this contradicts the definition of $\Omega^U_\lambda$.

Now we have $\rho>R_1$, and hence both $r$ and $\rho$ are large, the uniform angular
version of \eqref{eq:U-tail} and the inequality $U(x)>U(x^\lambda)$ imply
\begin{equation*}
 r^{4-n}(1+C/r)>\rho^{4-n}(1-C/\rho).
\end{equation*}
After increasing $R_0$,
\begin{equation*}
 1\leq\frac r\rho\leq1+\frac C\rho,
 \qquad
 C^{-1}\leq\frac{U_\lambda(x)}{U(x)}\leq C.
\end{equation*}
Consequently $c_{11}=O(r^{-4})$.  Moreover, every intermediate value between
$U$ and $U_\lambda$ is comparable with $r^{4-n}$.  This is essential when
$n\geq9$, because then $q-2<0$. Consequently,
\begin{align*}
 V\eta_1^{q-2}
 &=O(r^{2-n})O\bigl(r^{(4-n)(q-2)}\bigr)=O(r^{-6}),\\
 \eta_2^{2q-2}&=O(r^{-8}).
\end{align*}
This proves the estimates for $c_{11}$ and $c_{21}$ in every dimension
$n\geq5$.

On $\Omega^V_\lambda$, the same argument applies: the positive
representation \eqref{eq:V-Bocher} gives a uniform lower bound for $V$ on
each fixed punctured ball, while \eqref{eq:V-tail} controls the far region.
The inequality $V(x)>V(x^\lambda)$ and the two-sided form of
\eqref{eq:V-tail} first give
\begin{equation*}
 1\leq r/\rho\leq1+C/\rho.
\end{equation*}
This comparison of radii allows \eqref{eq:U-tail} to be used at the reflected point:
\begin{equation*}
 U_\lambda^{q-1}
 =O\bigl(\rho^{(4-n)(q-1)}\bigr)=O(r^{-4}).
\end{equation*}
Thus $c_{22}=O(r^{-4})$. 

\eqref{eq:tail-norm-small} now follows immediately, completing the proof of the claim.

If $\lambda\ll0$, then $|x|\geq|\lambda|$ on $\Sigma_\lambda$.  Moreover,
\begin{equation*}
 D_\lambda\leq U\in L^{2q}(\Sigma_\lambda),
 \qquad
 E_\lambda\leq V\in L^2(\Sigma_\lambda)
\end{equation*}
by \eqref{eq:U-tail}--\eqref{eq:V-tail}.  It follows from
\eqref{eq:tail-norm-small} that the three coefficient norms in
\eqref{eq:HLS-smallness} are smaller than $\varepsilon_0$ for all
sufficiently negative $\lambda$.  Step~1 therefore gives
\begin{equation*}
 U_\lambda\geq U,
 \qquad V_\lambda\geq V
 \quad\hbox{in }\Sigma_\lambda
\end{equation*}
for every sufficiently negative $\lambda$.

\medskip
\noindent\emph{Step 3. The limiting position and the strong alternative.}
By Step~2, the set in the following definition is nonempty, while $0$ is an
upper bound:
\begin{equation*}
 \lambda_0
 =
 \sup\left\{
 \lambda<0:
 \begin{array}{l}
 U_\mu\geq U,\ V_\mu\geq V\ \hbox{in }\Sigma_\mu, 
 \text{ for every }\mu\leq\lambda
 \end{array}
 \right\}.
\end{equation*}
Hence $\lambda_0$ is well defined.  By pointwise continuity away from the
reflected puncture,
\begin{equation*}
 U_{\lambda_0}\geq U,
 \qquad
 V_{\lambda_0}\geq V
 \quad\hbox{in }\Sigma_{\lambda_0}\setminus\{P_{\lambda_0}\}.
\end{equation*}

Put
\begin{equation*}
 \varphi=U_{\lambda_0}-U,
 \qquad
 \psi=V_{\lambda_0}-V.
\end{equation*}
By \eqref{eq:exact-reflection-U} and \eqref{eq:exact-reflection-V},
$\varphi,\psi\geq0$ are represented by the positive kernel
$\mathcal K_{\lambda_0}$ applied to the corresponding nonnegative nonlinear
differences; moreover, $d[\Phi(x^{\lambda_0})-\Phi(x)]\geq0$.

Since both $H_1$ and $H_2$ are strictly increasing in each positive
variable, either
\begin{equation}
 \varphi\equiv\psi\equiv0,
 \qquad\hbox{or}\qquad
 \varphi>0,\ \psi>0
 \quad\hbox{in }\Sigma_{\lambda_0}\setminus\{P_{\lambda_0}\}.
 \label{eq:strong-alternative}
\end{equation}
Indeed, suppose first that $\varphi\not\equiv0$.  Then
$\varphi>0$ on a set of positive measure.  The inequality
$H_2(U_{\lambda_0},V_{\lambda_0})>H_2(U,V)$ holds on this set because
\(H_2\) is strictly increasing in its first variable and
\(V_{\lambda_0}\ge V\).  Hence \eqref{eq:exact-reflection-V}, together
with \(\mathcal K_{\lambda_0}>0\) and
$d\bigl[\Phi(x^{\lambda_0})-\Phi(x)\bigr]\ge0$, gives
\begin{equation*}
\psi>0
\qquad\text{in }\Sigma_{\lambda_0}\setminus\{P_{\lambda_0}\}.
\end{equation*}
Since \(H_1\) is strictly increasing in its second variable, \eqref{eq:exact-reflection-U} then
gives
\begin{equation*}
\varphi>0
\qquad\text{in }\Sigma_{\lambda_0}\setminus\{P_{\lambda_0}\}.
\end{equation*}
The same argument with \(\varphi\) and \(\psi\) interchanged applies if
\(\psi\not\equiv0\). Therefore either
\begin{equation*}
\varphi\equiv\psi\equiv0,
\end{equation*}
or
\begin{equation*}
\varphi>0,\qquad \psi>0
\qquad\text{in }\Sigma_{\lambda_0}\setminus\{P_{\lambda_0}\}.
\end{equation*}

\medskip
\noindent\emph{Step 4. The plane cannot stop strictly at a negative
position.}
Assume $\lambda_0<0$ and that the strict alternative in
\eqref{eq:strong-alternative} holds.  We show that the contraction criterion
of Step~1 remains valid immediately to the right of $\lambda_0$.
For every fixed $\lambda<0$, the puncture lies outside $\Sigma_\lambda$,
and
\begin{equation*}
     D_\lambda\leq U,\qquad E_\lambda\leq V.
\end{equation*}
Thus \eqref{eq:U-tail}--\eqref{eq:V-tail}, together with local smoothness,
give
\begin{equation*}
    D_\lambda\in L^{2q}(\Sigma_\lambda),
 \qquad
 E_\lambda\in L^2(\Sigma_\lambda). 
\end{equation*}

Now choose $R$ so large that the uniform tail contribution in
\eqref{eq:tail-norm-small} is below $\varepsilon_0/2$.  We claim that, for
this fixed $R$, as $\lambda\downarrow\lambda_0$,
\begin{equation*}
 \|c_{11}\|_{L^{n/2}(\Omega^U_\lambda\cap B_R)}
 +\|c_{22}\|_{L^{n/2}(\Omega^V_\lambda\cap B_R)}
 +\|c_{21}\|_{L^{n/4}(\Omega^U_\lambda\cap B_R)}
 \longrightarrow0.
\end{equation*}

By the strict inequalities at \(\lambda_0\) and continuity in \(\lambda\),
the compact bad sets shrink to a null set; hence
\begin{equation*}
|\Omega_\lambda^U\cap B_R|
+
|\Omega_\lambda^V\cap B_R|
\longrightarrow0
\qquad\text{as }\lambda\downarrow\lambda_0.
\end{equation*}
It remains only to bound the coefficients uniformly on the compact bad
sets.  The possible difficulty is that the reflected point may approach the
puncture, and, when \(q-2<0\), the factor in \(c_{21}\) could in principle
become singular. 

The exact positive representation $U=I_2H_1(U,V)$
gives, for every fixed \(R'>0\),
\begin{equation*}
U(x)\ge c(R')>0
\qquad (0<|x|\le R').
\end{equation*}
Hence, on \(\Omega_\lambda^U\cap B_R\), both \(U\) and \(U_\lambda\)
remain between positive constants, uniformly for
\(\lambda\downarrow\lambda_0\).  Thus \(c_{11}\) and \(c_{21}\) are
uniformly bounded there.  On \(\Omega_\lambda^V\cap B_R\),
$V_\lambda<V\le C_R$ and $V_\lambda\ge AU_\lambda^q$,
so \(U_\lambda\le C_R\), and hence \(c_{22}\) is uniformly bounded as well.

Together with the vanishing measure of the compact bad sets, the absolute
continuity gives
\begin{equation*}
\|c_{11}\|_{L^{n/2}(\Omega_\lambda^U\cap B_R)}
+
\|c_{22}\|_{L^{n/2}(\Omega_\lambda^V\cap B_R)}
+
\|c_{21}\|_{L^{n/4}(\Omega_\lambda^U\cap B_R)}
\longrightarrow0.
\end{equation*}
Combining this with the far-field estimate \eqref{eq:tail-norm-small}, the criterion of Step~1 applies
for every \(\lambda>\lambda_0\) sufficiently close to \(\lambda_0\), which
contradicts the definition of \(\lambda_0\).

\medskip
\noindent\emph{Step 5. The puncture and the conclusion.}
If $\lambda_0<0$, Step~4 leaves the only alternative:
\begin{equation*}
 U_{\lambda_0}\equiv U,
 \qquad
 V_{\lambda_0}\equiv V.
\end{equation*}
Reflection across $x_1=\lambda_0$ sends the puncture $0$ to the regular point
\begin{equation*}
 P_{\lambda_0}=2\lambda_0e_1.
\end{equation*}
For $y$ near $0$, exact symmetry identifies $(U,V)$ at $y$ with the
reflection of the smooth pair near $P_{\lambda_0}$.  Thus $(U,V)$ extends
smoothly across $0$, contradicting the nonremovable branch under
consideration.  Therefore $\lambda_0=0$.

Repeating the argument from the opposite direction gives symmetry across
$\{x_1=0\}$.  Since the coordinate direction was arbitrary, both
$U$ and $V$ are radial about $0$.
\end{proof}

\section{Spherical compactification and the admissible classification}
Under the fast decay condition, we can classify the solutions.
\begin{proposition}
\label{prop:fast-bubble}
Suppose $u$ satisfies the hypotheses of Theorem~\ref{thm:main} and
\begin{equation*}
 u(x)\leq C|x|^{4-n}
 \qquad(|x|\gg1).
\end{equation*}
Then $u$ is one of the bubbles \eqref{eq:bubble}.
\end{proposition}

\begin{proof}
The Kelvin transform $U$ is bounded near zero.  The Kelvin pair $(U,V)$ was
constructed in Section~4, so the bounded
removability statement in Lemma~\ref{lem:Bocher-pair} applies directly: both
$U$ and $V$, and hence $Z=V+AU^q$, extend smoothly and positively across the
origin.

Stereographic compactification now produces a smooth positive conformal factor
on $\Sn$.  The scalar-curvature quantity is nonnegative everywhere and not
identically zero; the $Q$-curvature quantity is likewise semipositive.  The
Gursky--Malchiodi maximum principle, as recalled in Theorem 2.1 of
\cite{GeWangWei26}, upgrades the scalar curvature to strict positivity.  Theorem
3.2 of \cite{GeWangWei26} shows that the metric is a constant multiple of a
M\"obius pullback of the round metric.  In our normalization, this gives exactly
\eqref{eq:bubble} in Euclidean coordinates.
\end{proof}

\subsection{Completion of the admissible classification}

\begin{proof}[Proof of Theorem~\ref{thm:main}]
If $w\equiv0$, then $u^p$ is a positive entire harmonic function, and hence
$u$ is constant.  We may therefore assume $w\not\equiv0$, as was done from
Section~2 onward.

Fix two distinct points $a,b\in\R^n$.  The Kelvin
construction of Section~4 and Proposition~\ref{prop:moving-planes} may be
applied with either point as the inversion center.  For
$\xi\in\{a,b\}$, write
\begin{equation*}
 U_\xi(x)
 =|x|^{4-n}u\left(\xi+\frac{x}{|x|^2}\right),
 \qquad x\in\R^n\setminus\{0\},
\end{equation*}
and let $V_\xi$ be the corresponding auxiliary function.

Suppose first that the puncture is removable for one of the pairs
$(U_\xi,V_\xi)$.  Then $U_\xi$ is bounded near zero and, for $y\ne\xi$,
\begin{equation*}
 u(y)=|y-\xi|^{4-n}
 U_\xi\left(\frac{y-\xi}{|y-\xi|^2}\right).
\end{equation*}
It follows that $u(y)=O(|y|^{4-n})$, as $|y|\to\infty$. Proposition~\ref{prop:fast-bubble} therefore gives a standard bubble.

It remains to rule out the case in which both punctures are nonremovable.
Proposition~\ref{prop:moving-planes} then implies that $U_a$ and $U_b$ are
radial about zero.  Therefore, $u$ is radial about both $a$ and $b$.  If $x$ does not lie on the line through
$a$ and $b$, then
\begin{equation*}
 \nabla u(x)\parallel x-a,
 \qquad
 \nabla u(x)\parallel x-b.
\end{equation*}
The two directions are not parallel, so $\nabla u(x)=0$.  Since the complement
of that line is dense and $u\in C^1(\R^n)$, it follows that
$\nabla u\equiv0$, contradicting $w\not\equiv0$.

Thus at least one of the two Kelvin punctures is removable, and the preceding
argument proves that $u$ is a standard bubble.  Under the strict condition
$w>0$, the constant alternative is excluded.
\end{proof}

\section{\texorpdfstring{Automatic admissibility from $-\Delta u\geq0$}
{Automatic admissibility from weak superharmonicity}}
\label{sec:weak-superharmonic-upgrade}

We now prove Theorem~\ref{thm:weak-superharmonic}.  This section does not
assume a sign for $w=-\Delta(u^p)$.  We recover the positivity of $w$ for
nonconstant solutions from $-\Delta u\geq0$.

We first introduce the following standard local \(L^s\)-to-\(L^\infty\)
estimate for nonnegative subharmonic functions, including the range
\(0<s<1\).
\begin{lemma}[Local $L^s$ bound for subharmonic functions]
\label{lem:subharmonic-small-p}
If $f\geq0$ is distributionally subharmonic in $B_R$ and $s>0$, then
\begin{equation}
 \sup_{B_{R/2}}f
 \leq C_{n,s}\left(\fint_{B_R}f^s\right)^{1/s}.
 \label{eq:subharmonic-small-p}
\end{equation}
Here and below a distributionally subharmonic function is identified with its
upper-semicontinuous representative.
\end{lemma}

\begin{proof}
For $s\geq1$ this is the usual local subsolution estimate.  For $0<s<1$,
let $M(\rho)=\sup_{B_\rho}f$.  The mean-value inequality on balls of radius
comparable to $t-r$ gives, for $R/2\leq r<t\leq R$,
\begin{align*}
 M(r)&\leq C(t-r)^{-n}\int_{B_t}f\\
     &\leq C(t-r)^{-n}M(t)^{1-s}\int_{B_R}f^s.
\end{align*}
Young's inequality absorbs one half of $M(t)$; iteration over radii
then gives \eqref{eq:subharmonic-small-p}.  Convolution with standard
mollifiers on slightly smaller balls gives the same
statement for distributional subsolutions.
\end{proof}
Now we prove that $T>0$ under the weak condition $-\Delta u\geq 0$.
\begin{lemma}
\label{lem:weak-T-normality}
Assume only that $u>0$ solves \eqref{eq:main} and
$z=-\Delta u\geq0$.  If $u$ is nonconstant, then
\begin{equation}
 \inf_{\R^n}u=0,
 \qquad u=I_2z,
 \qquad T:=z-Ku^q>0.
 \label{eq:weak-T-positive}
\end{equation}
Moreover, for every fixed center $x_0$ and $R\geq1$,
\begin{equation}
 \int_{B_R(x_0)}u\leq C_{x_0}R^{(n+4)/2},
 \qquad
 \int_{B_R(x_0)}z\leq C_{x_0}R^{n/2}.
 \label{eq:weak-scale-bounds}
\end{equation}
\end{lemma}

\begin{proof}
The general Riesz argument in Lemma~\ref{lem:riesz-decomposition}, applied
only to the positive superharmonic function $u$, gives
\begin{equation*}
 u=\ell+I_2z,
 \qquad \ell=\inf_{\R^n}u\geq0.
\end{equation*}
By \eqref{eq:Newton-shell-mean}, the spherical mean of $u$ about every fixed
center tends to $\ell$.

A direct computation gives
\begin{equation*}
 -\Delta T
 =qK(p-1)u^{q-2}|\nabla u|^2
 \geq0.
\end{equation*}

Kato's positive-part inequality shows that $T^-$ is a
nonnegative distributionally subharmonic function.  Since $z\geq0$,
\begin{equation*}
 0\leq T^-\leq Ku^q.
\end{equation*}
Applying Lemma~\ref{lem:subharmonic-small-p} with exponent
$1/q\in(0,1)$ and using the superharmonicity of $u$, we obtain
\begin{align*}
 \sup_{B_{R/2}}T^-
 &\leq C\left(\fint_{B_R}(T^-)^{1/q}\right)^q\\
 &\leq CK\left(\fint_{B_R}u\right)^q\\
 &\leq CKu(0)^q.
\end{align*}
Letting $R\to\infty$ gives $M:=\sup_{\R^n}T^-<\infty$.

The function $M-T^-$ is nonnegative and superharmonic, with $\inf (M-T^-)=0$.
Its global Riesz decomposition therefore has no positive harmonic constant,
so \eqref{eq:Newton-shell-mean} gives
\begin{equation*}
 \fint_{\partial B_R}(M-T^-)\longrightarrow0.
\end{equation*}
Because $0<1/q<1$,
\begin{equation*}
 0\leq M^{1/q}-(T^-)^{1/q}\leq(M-T^-)^{1/q}.
\end{equation*}
Concavity and Jensen's inequality show that the spherical mean of
$(M-T^-)^{1/q}$ tends to zero.  Hence the spherical mean of $(T^-)^{1/q}$ tends to
$M^{1/q}$.  But $(T^-)^{1/q}\leq K^{1/q}u$ and the spherical mean of $u$
tends to $\ell$, whence
\begin{equation}
 M\leq K\ell^q.
 \label{eq:weak-M-bound}
\end{equation}
It follows that
\begin{equation*}
 z=T+Ku^q\geq-M+Ku^q\geq K(u^q-\ell^q).
\end{equation*}

If $u$ is nonconstant and $\ell>0$, then
$h=u-\ell=I_2z>0$ and convexity gives
\begin{equation*}
 -\Delta h=z\geq Kq\ell^{q-1}h.
\end{equation*}
Barta's inequality on $B_R$ would imply
$\lambda_1(B_R)\geq Kq\ell^{q-1}>0$ for every $R$, contrary to
$\lambda_1(B_R)\to0$.  Therefore $\ell=0$, and
\eqref{eq:weak-M-bound} gives $M=0$.  Thus $T\geq0$.  If $T$ vanished at one
point, the strong maximum principle would give
$T\equiv0$ and then $\nabla u\equiv0$, contrary to nonconstancy.  This proves
\eqref{eq:weak-T-positive}.

It remains to prove the scale bounds.  Let $f(R)$ be the spherical mean of
$u$ about $x_0$.  From $z\geq Ku^q$, Jensen's inequality gives
\begin{equation*}
 -\bigl(R^{n-1}f'(R)\bigr)'
 \geq KR^{n-1}f(R)^q.
\end{equation*}
The function $f$ is decreasing, so integration from $0$ to $R$ yields
\begin{equation*}
 -f'(R)\geq \frac Kn Rf(R)^q.
\end{equation*}
Consequently
\begin{equation}
 f(R)\leq C_{x_0}R^{-2/(q-1)}
 =C_{x_0}R^{-(n-4)/2}.
 \label{eq:weak-spherical-decay}
\end{equation}
If $\mathcal M(R)=R^{n-1}(-f'(R))$, then $\mathcal M$ is nondecreasing and
\begin{align*}
 f(R)-f(2R)
 &=\int_R^{2R}\mathcal M(t)t^{1-n}\dd t\\
 &\geq cR^{2-n}\mathcal M(R).
\end{align*}
Since $\int_{B_R(x_0)}z=|\mathbb S^{n-1}|\mathcal M(R)$,
\eqref{eq:weak-spherical-decay} gives the second estimate in
\eqref{eq:weak-scale-bounds}; integrating $f$ gives the first.
\end{proof}
Now we improve the estimate of $z$ by a gradient term in order to prove the
positivity of $w$.
\begin{proposition}
\label{prop:weak-gradient-upgrade}
Every nonconstant solution in Lemma~\ref{lem:weak-T-normality} satisfies
\begin{equation}
 z\geq Ku^q+(p-1)\frac{|\nabla u|^2}{u}.
 \label{eq:weak-gradient-upgrade}
\end{equation}
Consequently $w\geq pK u^{p+q-1}>0$.
\end{proposition}

\begin{proof}
For this proof, set
\begin{equation*}
 \alpha:=p-1=\frac{2}{n-4}.
\end{equation*}
The equation expanded without any sign assumption on $w$ is
\begin{equation}
 \Delta^2u+qK u^{q-1}\Delta u
 +qK\alpha u^{q-2}|\nabla u|^2=0.
 \label{eq:weak-fourth-expanded}
\end{equation}
Fix $\varepsilon>0$, put $v=u+\varepsilon$, and define
\begin{equation*}
 w_b=\Delta u+b\frac{|\nabla u|^2}{v}+Ku^q,
 \qquad b\geq0.
\end{equation*}
The inequality $w_b\leq0$ is equivalent to
\begin{equation*}
 z\geq Ku^q+b\frac{|\nabla u|^2}{v}.
\end{equation*}
Thus the desired estimate is obtained by increasing the coefficient $b$ from
$0$ to $\alpha$.  Lemma~\ref{lem:weak-T-normality} gives the starting point
$w_0=-T<0$.  We shall bootstrap $w_b\leq0$ by eliminating the positive set of $w_b$ with a global Liouville argument.

\smallskip
\noindent\emph{Claim 1.}
Suppose $0\leq c<\alpha$ and $w_c\leq0$.  Then there exists
$b\in(c,\alpha)$ such that
\begin{equation*}
 \Delta(v^{-b}w_b)\geq0
 \qquad\hbox{on }\Omega_b:=\{w_b>0\}.
\end{equation*}

To prove the claim, take for the moment an arbitrary $b\in(c,\alpha)$.  The weight
$v^{-b}$ naturally leads to the drift operator
\begin{equation*}
 L_b=\Delta-\frac{2b}{v}\nabla u\cdot\nabla,
\end{equation*}
which also cancels the third-order terms in the Bochner calculation.
Bochner's identity and \eqref{eq:weak-fourth-expanded} give the exact formula
\begin{equation}
\begin{aligned}
 L_bw_b={}&\frac{2b}{v}
 \left|D^2u-(1+b)\frac{du\otimes du}{v}\right|^2
 -\frac{2b^2(1+b)|\nabla u|^4}{v^3}
 -\frac{b|\nabla u|^2}{v^2}\Delta u \\
 &\quad+qK\left(\alpha-2b\frac uv\right)u^{q-2}|\nabla u|^2,
\end{aligned}
\label{eq:weak-Lbwb}
\end{equation}
and direct differentiation of the weight gives
\begin{equation}
 \Delta(v^{-b}w_b)
 =v^{-b}\left\{
 L_bw_b-\frac{bw_b}{v}\Delta u
 +\frac{b(b+1)w_b}{v^2}|\nabla u|^2
 \right\}.
 \label{eq:weak-weighted-lap}
\end{equation}

On $\Omega_b$, the preceding iterate gives
\begin{equation}
 0<\frac{w_b}{v}
 \leq(b-c)\frac{|\nabla u|^2}{v^2},
 \label{eq:weak-positive-set-bound}
\end{equation}
and
\begin{equation*}
 \operatorname{tr}\left(D^2u-(1+b)\frac{du\otimes du}{v}\right)
 \leq-(1+b+c)\frac{|\nabla u|^2}{v}-Ku^q<0.
\end{equation*}
Using $|A|^2\geq(\operatorname{tr}A)^2/n$ in
\eqref{eq:weak-Lbwb}, substituting
$\Delta u/v=w_b/v-b|\nabla u|^2/v^2-Ku^q/v$ in
\eqref{eq:weak-weighted-lap}, and collecting terms gives
\begin{equation}
\begin{aligned}
 v^{b-1}\Delta(v^{-b}w_b)
 \geq{}&
 \left\{\frac{2b}{n}(1+b+c)^2-b^2(1+2b)\right\}
 \frac{|\nabla u|^4}{v^4}
 +\frac{2bK^2}{n}\frac{u^{2q}}{v^2} \\
 &+D_b\left(\frac uv\right)\frac{u^q|\nabla u|^2}{v^3}+b\frac{w_b}{v}\left(
 2b\frac{|\nabla u|^2}{v^2}
 +K\frac{u^q}{v}-\frac{w_b}{v}\right),
\end{aligned}
\label{eq:weak-key-lower}
\end{equation}
where
\begin{equation*}
 D_b(\theta)=\frac{4bK(1+b+c)}{n}+bK
 +qK\frac{\alpha-2b\theta}{\theta^2},
 \qquad 0<\theta\leq1.
\end{equation*}

We now choose $b\in(c,\alpha)$ so that the
quartic-gradient coefficient vanishes, namely
\begin{equation}
 \frac2n(1+b+c)^2=b(1+2b).
 \label{eq:weak-b-recurrence}
\end{equation}
Easy to verify that such a $b$ exists and is unique. 

For this choice of $b$, the first term in \eqref{eq:weak-key-lower} vanishes,
while \eqref{eq:weak-positive-set-bound} gives
\begin{equation*}
 2b\frac{|\nabla u|^2}{v^2}
 +K\frac{u^q}{v}-\frac{w_b}{v}
 \geq(b+c)\frac{|\nabla u|^2}{v^2}
 +K\frac{u^q}{v}>0.
\end{equation*}
It remains to verify \(D_b(u/v)>0\). Since \(b<\alpha\),
\begin{equation*}
 \frac{d}{d\theta}
 \left(\frac{\alpha-2b\theta}{\theta^2}\right)
 =
 \frac{2(b\theta-\alpha)}{\theta^3}<0,
 \qquad 0<\theta\le1.
\end{equation*}
Hence \(D_b(u/v)\ge D_b(1)\). By \eqref{eq:weak-b-recurrence} and
\(b<\alpha=2/(n-4)\),
\begin{equation*}
 \frac{4b(1+b+c)}{n}
 =\frac{4b\sqrt{b(1+2b)}}{\sqrt{2n}}>2b^2.
\end{equation*}
Therefore,
\begin{align*}
 \frac{D_b(1)}{K}>2b^2+b+q(\alpha-2b)
 =(\alpha-b)(1+2\alpha-2b)>0,
\end{align*}
where $q=1+2\alpha$.
All terms on the right of \eqref{eq:weak-key-lower} are therefore
nonnegative, proving Claim~1.

\smallskip
\noindent\emph{Claim 2.}
For the $b$ supplied by Claim~1, one has $w_b\leq0$ in $\R^n$.

Claim~1 and Kato's positive-part inequality give
\begin{equation*}
 \Delta(v^{-b}w_b)_+\geq0
 \qquad\hbox{in }\mathcal D'(\R^n).
\end{equation*}
Thus $(v^{-b}w_b)_+$ is a nonnegative entire subharmonic function.  On its
positive set, \eqref{eq:weak-positive-set-bound} implies
\begin{equation}
 \bigl[(v^{-b}w_b)_+\bigr]^{1/2}
 \leq C_{b,c,\varepsilon}|\nabla u|.
 \label{eq:weak-positive-part-gradient}
\end{equation}

We next use entire-space growth, which is the genuinely global part of the
argument.  Since $u=I_2z$, decompose for $R\geq1$
\begin{equation*}
 \mathcal P_R=I_2(z\mathbf1_{B_{2R}}),
 \qquad
 H_R=u-\mathcal P_R=I_2(z\mathbf1_{\R^n\setminus B_{2R}}).
\end{equation*}
The second function is positive and harmonic in $B_{2R}$.  Integrability of
the gradient of the Newton kernel and the interior gradient estimate for a
positive harmonic function give, more explicitly,
\begin{align*}
 \int_{B_R}|\nabla \mathcal P_R(x)|\,\dd x
 \leq C\int_{B_{2R}}z(y)
       \int_{B_R}|x-y|^{1-n}\,\dd x\,\dd y\leq CR\int_{B_{2R}}z,
\end{align*}
whereas the local $L^1$ gradient estimate for harmonic functions yields
\begin{align*}
 \int_{B_R}|\nabla H_R|
\leq CR^{-1}\int_{B_{2R}}H_R\leq CR^{-1}\int_{B_{2R}}u.
\end{align*}
Consequently,
\begin{equation*}
 \int_{B_R}|\nabla u|
 \leq C\left(
 R\int_{B_{2R}}z+R^{-1}\int_{B_{2R}}u
 \right)
 \leq CR^{(n+2)/2}=o(R^n).
\end{equation*}
For any fixed \(x\in\mathbb R^n\), take \(R>2|x|\). Applying
Lemma~\ref{lem:subharmonic-small-p} to \((v^{-b}w_b)_+\) on \(B_R(0)\)
with \(s=1/2\), and using \eqref{eq:weak-positive-part-gradient}, gives
\begin{align*}
 (v^{-b}w_b)_+(x)
 &\leq C
 \left(R^{-n}\int_{B_R}\bigl[(v^{-b}w_b)_+\bigr]^{1/2}\right)^2
 \leq C_{\varepsilon,b,c}
 \left(R^{-n}\int_{B_R}|\nabla u|\right)^2\\
 &\leq C_{\varepsilon,b,c}R^{-(n-2)}
 \longrightarrow0.
\end{align*}
Hence \((v^{-b}w_b)_+(x)=0\). Since \(x\) was arbitrary,
\((v^{-b}w_b)_+\equiv0\), and therefore \(w_b\leq0\), proving Claim~2.

\smallskip
Next we will do the iteration on $b$.
Starting from $b_0=0$, choose $b_{j+1}\in(b_j,\alpha)$ from
\eqref{eq:weak-b-recurrence} with $c=b_j$.  Claims~1 and~2 show inductively
that $w_{b_j}\leq0$ for every $j$, while
\begin{equation*}
 0=b_0<b_1<b_2<\cdots<\alpha.
\end{equation*}
Thus $b_j\to\bar b\leq\alpha$.  Passing to the limit in the defining relation
gives
\begin{equation*}
 \frac2n(1+2\bar b)^2=\bar b(1+2\bar b),
\end{equation*}
and hence $\bar b=2/(n-4)=\alpha$.  For fixed $\varepsilon$, let $j\to\infty$
in $w_{b_j}\leq0$ and then let $\varepsilon\downarrow0$.  We obtain
\begin{equation*}
 \Delta u+(p-1)\frac{|\nabla u|^2}{u}+Ku^q\leq0,
\end{equation*}
which is \eqref{eq:weak-gradient-upgrade}.  Finally,
\begin{equation*}
 w=p u^{p-1}\left(z-(p-1)\frac{|\nabla u|^2}{u}\right)
 \geq pK u^{p+q-1}>0.
\end{equation*}
\end{proof}

\begin{proof}[Proof of Theorem~\ref{thm:weak-superharmonic}]
Assume $u$ is not a constant.
Proposition~\ref{prop:weak-gradient-upgrade} gives $w>0$, and
Theorem~\ref{thm:main} gives the standard bubble.
\end{proof}

\section{Non-admissible radial scattering solutions}
\label{sec:nonadmissible}

This section studies positive solutions of \eqref{eq:main} without assuming
$w=-\Delta(u^p)>0$.  The conclusion is different: the constant
solutions are accompanied by a family of positive
radial scattering solutions.  The construction below is rigorous and does not
use the normal-potential representation from the admissible part of the paper.

\subsection{An exact second-order reduction}

Keep \(z=-\Delta u\) and set
\begin{equation*}
 \tau=-T=Ku^q-z.
\end{equation*}
We obtain
\begin{equation}
 \Delta u=\tau-Ku^q,\qquad
 \Delta\tau=c_np(p-1)u^{q-2}|\nabla u|^2.
 \label{eq:nonadm-system}
\end{equation}
Conversely, eliminating \(\tau\) from \eqref{eq:nonadm-system} recovers
\eqref{eq:main}. Thus \eqref{eq:nonadm-system} is an exact
second-order formulation of the original equation.

For a radial pair, \eqref{eq:nonadm-system} becomes
\begin{align}
 u''+\frac{n-1}{r}u'&=\tau-Ku^q,                          \label{eq:nonadm-radial-u}\\
 \tau''+\frac{n-1}{r}\tau'&=\cN p(p-1)u^{q-2}(u')^2.    \label{eq:nonadm-radial-tau}
\end{align}
Normalize the regular initial data by
\begin{equation}
 u(0)=1,\quad u'(0)=0,\quad u''(0)=-s,\qquad
 \tau(0)=K-ns,\quad\tau'(0)=0.                           \label{eq:nonadm-IVP}
\end{equation}
The Taylor expansion is
\begin{equation*}
 u(r)=1-\frac{s}{2}r^2+\frac{n-2}{32}s r^4+O(r^6).
\end{equation*}
Thus $s=0$ is the constant solution.  The normalized bubble corresponds to
$s=(n-4)/4$.

We first prove the standard local well-posedness and continuation
criterion for the radial system.
\begin{lemma}
\label{lem:nonadm-local-IVP}
Let $u_0>0$ and $\tau_0\in\mathbb R$.  The radial system
\eqref{eq:nonadm-radial-u}--\eqref{eq:nonadm-radial-tau}, with
\begin{equation*}
 u(0)=u_0,\qquad \tau(0)=\tau_0,\qquad u'(0)=\tau'(0)=0,
\end{equation*}
has a unique solution on a nontrivial interval $[0,r_0]$.  It is smooth and
even in $r$, and hence defines a smooth radial pair near the origin in
$\mathbb R^n$.  The solution continues uniquely as long as $u$ stays positive
and $(u,\tau,u',\tau')$ remains bounded on finite intervals.
\end{lemma}

\begin{proof}
Writing $G_1=\tau-Ku^q$ and
$G_2=\cN p(p-1)u^{q-2}(u')^2$, the regular solutions are precisely the
fixed points of
\begin{align*}
 u(r)&=u_0+\int_0^r t^{1-n}\int_0^t s^{n-1}G_1(s)\,\dd s\,\dd t,\\
 \tau(r)&=\tau_0+\int_0^r t^{1-n}\int_0^t s^{n-1}G_2(s)\,\dd s\,\dd t.
\end{align*}
On a small ball in $C^1([0,r_0])^2$ with $u$ bounded away from zero, this
Volterra map is a contraction.  The integral equations successively give
even Taylor expansions and smoothness.  Away from $r=0$ the standard ODE
continuation theorem gives the last assertion.
\end{proof}
Now we shall also use the following elementary asymptotic fact for the radial
Helmholtz equation.
\begin{lemma}
\label{lem:radial-Helmholtz-scattering}
Let $k>0$, $d=(n-1)/2$, and suppose
\begin{equation*}
 \rho''+\frac{n-1}{r}\rho'+k^2\rho=G\qquad(r\geq1).
\end{equation*}
If $P=r^d\rho$ and $P,P'$ are bounded, while
$\int_1^\infty r^d|G(r)|\,\dd r<\infty$, then there are $A,B\in\mathbb R$
such that
\begin{equation*}
 P(r)=A\cos(kr)+B\sin(kr)+o(1),
\end{equation*}
with the analogous differentiated expansion.  Moreover, the scattering
vector $(A,B)$ is bounded in terms of $|P(1)|+|P'(1)|$ and the displayed
$L^1$ norm.
\end{lemma}

\begin{proof}
The transformed equation is
\begin{equation*}
 P''+k^2P=r^dG+\frac{(n-1)(n-3)}{4r^2}P.
\end{equation*}
Its right-hand side belongs to $L^1(1,\infty)$.  Variation of constants gives
convergent sine and cosine coefficients, the stated estimate, and a remainder
bounded by the tail of that $L^1$ function.
\end{proof}

\subsection{Small radial scattering solutions}

The following proposition gives the normalized and quantitative form of
Theorem~\ref{thm:nonadm-scattering}.

\begin{proposition}
\label{prop:nonadm-scattering-precise}
For every $n\geq5$ there exists $\varepsilon_n>0$ such that, for every
$0<|s|<\varepsilon_n$, the initial problem
\eqref{eq:nonadm-radial-u}--\eqref{eq:nonadm-IVP} has a unique global smooth
solution with
\begin{equation*}
 u_s(r)>0\qquad(0\leq r<\infty).
\end{equation*}
There is a number $\ell_s>0$ such that $u_s(r)\to\ell_s$.  More precisely, if
\begin{equation*}
 d=\frac{n-1}{2},\qquad
 k_s^2=\cN p\,\ell_s^{2(p-1)},
\end{equation*}
then there are constants $A_s,B_s$, not both zero, and a function $L_s$ such
that
\begin{align}
 L_s(r)&\longrightarrow\ell_s,\qquad
 L_s(r)-\ell_s=O(r^{3-n}),                                \label{eq:nonadm-L-tail}\\
 u_s(r)&=L_s(r)+r^{-d}
 \{A_s\cos(k_sr)+B_s\sin(k_sr)+o(1)\}.                   \label{eq:nonadm-u-tail}
\end{align}
Furthermore,
\begin{align}
 -\Delta u_s
 &=k_s^2r^{-d}\{A_s\cos(k_sr)+B_s\sin(k_sr)+o(1)\},      \label{eq:nonadm-z-tail}\\
 -\Delta(u_s^p)
 &=p\ell_s^{p-1}k_s^2r^{-d}
 \{A_s\cos(k_sr)+B_s\sin(k_sr)+o(1)\}.                  \label{eq:nonadm-w-tail}
\end{align}
Consequently $-\Delta u_s$ and $-\Delta(u_s^p)$ change sign infinitely many
times.  In particular, $u_s$ is neither admissible nor a standard bubble.
\end{proposition}

\begin{proof}
Local existence and uniqueness follow from
Lemma~\ref{lem:nonadm-local-IVP}.  Only the uniform global estimate requires
work.  Once it is established, the
asymptotic formulas follow from standard one-dimensional asymptotic integration;
we include the details needed to verify all decay exponents.

For $|s|$ small, $\tau(0)=K-ns$ is positive.  As long as $\tau>0$, set
\begin{equation*}
 L=\left(\frac{\tau}{K}\right)^{1/q},\qquad e=u-L,\qquad
 k(r)^2=KqL^{q-1}.
\end{equation*}
Since $KL^q=\tau$, equations \eqref{eq:nonadm-radial-u} and
\eqref{eq:nonadm-radial-tau} give
\begin{equation*}
 e''+\frac{n-1}{r}e'+k(r)^2e=f,
\end{equation*}
where
\begin{equation*}
 f=-K\{(L+e)^q-L^q-qL^{q-1}e\}
   -\left(L''+\frac{n-1}{r}L'\right).
\end{equation*}
Since $L=(\tau/K)^{1/q}$, the chain rule and
\eqref{eq:nonadm-radial-tau} yield
\begin{equation*}
 L''+\frac{n-1}{r}L'
 =\frac{L}{q\tau}\,\cN p(p-1)u^{q-2}(u')^2
 -\frac{q-1}{q^2}\frac{L}{\tau^2}(\tau')^2.
\end{equation*}
Thus, whenever $u,L$ stay in a fixed compact subinterval of $(0,\infty)$,
\begin{equation}
 |f|\leq C_0\{e^2+(e')^2+(\tau')^2\}.                    \label{eq:nonadm-f-bound}
\end{equation}

We now close a continuation estimate without using the same bootstrap constant
at the linear and quadratic scales.  Put $\delta=|s|$.  Let $M_e,M_\tau>1$ be
constants to be chosen in that order.  On a maximal interval $[0,R]$, assume
\begin{align}
 |e(r)|+|e'(r)|&\leq M_e\delta(1+r)^{-d},                  \label{eq:nonadm-bootstrap-e}\\
 |\tau'(r)|&\leq M_\tau\delta^2(1+r)^{2-n},               \label{eq:nonadm-bootstrap-tau}
\end{align}
and $1/2\leq u,L\leq3/2$.  Radial integration of
\eqref{eq:nonadm-radial-tau} gives
\begin{equation}
 \tau'(r)=\cN p(p-1)r^{1-n}\int_0^r
 t^{n-1}u(t)^{q-2}u'(t)^2\,dt.                            \label{eq:nonadm-tau-flux}
\end{equation}
Because $u'=e'+L\tau'/(q\tau)$ and $2d=n-1$, the bootstrap assumptions imply
\begin{equation}
 |\tau'(r)|\leq C_1\bigl(M_e^2\delta^2+M_\tau^2\delta^4\bigr)
 \begin{cases}
 r,&0\leq r\leq1,\\
 r^{2-n},&r\geq1.
 \end{cases}                                               \label{eq:nonadm-tau-improve}
\end{equation}
In particular,
\begin{equation}
 \int_0^R|\tau'(r)|\,dr
 \leq C_2\bigl(M_e^2\delta^2+M_\tau^2\delta^4\bigr).    \label{eq:nonadm-tau-L1}
\end{equation}
Together with $\tau(0)=K-ns$, this keeps $\tau$, and hence $L$, uniformly close
to one once $\delta$ is small.  It also gives
$|(k^2)'|\leq C|\tau'|$.

To improve \eqref{eq:nonadm-bootstrap-e}, put $E=r^de$.  For $r\geq1$,
\begin{equation}
 E''+\left(k(r)^2-\frac{(n-1)(n-3)}{4r^2}\right)E=r^df.   \label{eq:nonadm-E}
\end{equation}
By \eqref{eq:nonadm-f-bound} and the bootstrap estimates,
\begin{equation*}
 \int_1^R r^d|f(r)|\,dr
 \leq C_3\bigl(M_e^2\delta^2+M_\tau^2\delta^4\bigr)
 \int_1^\infty r^{-d}\,dr
 \leq C_4\bigl(M_e^2\delta^2+M_\tau^2\delta^4\bigr),
\end{equation*}
where $d\geq2$ is exactly where $n\geq5$ is used.  The energy
\begin{equation*}
 \mathcal E(r)=\frac12(E')^2+\frac12k(r)^2E^2
\end{equation*}
obeys
\begin{align*}
 |\mathcal E'|
 &\leq C\bigl(|(k^2)'|+r^{-2}\bigr)\mathcal E
      +r^d|f|\sqrt{2\mathcal E}.
\end{align*}
The coefficient of $\mathcal E$ is integrable uniformly in $R$.  Gronwall's
inequality and \eqref{eq:nonadm-tau-L1} therefore give
\begin{equation}
 \sup_{1\leq r\leq R}(|E(r)|+|E'(r)|)
 \leq C_5\delta
 +C_6\bigl(M_e^2\delta^2+M_\tau^2\delta^4\bigr).         \label{eq:nonadm-E-bound}
\end{equation}
The regular ODE estimate on $[0,1]$ supplies the initial term $C_5|s|$.
Choose first $M_e>4C_5$, then
$M_\tau>4C_1M_e^2$.  Finally choose $\delta>0$ so small that
\begin{align*}
 C_1M_\tau^2\delta^2<\frac14M_\tau,\qquad
 C_6\bigl(M_e^2\delta+M_\tau^2\delta^3\bigr)<\frac14M_e,
\end{align*}
and so that the total variation in \eqref{eq:nonadm-tau-L1}, together with
$|\tau(0)-K|=n\delta$, is smaller than a fixed neighborhood of $K$.
Then \eqref{eq:nonadm-tau-improve} improves
\eqref{eq:nonadm-bootstrap-tau}, while
\eqref{eq:nonadm-E-bound} and the regular estimate on $[0,1]$ improve
\eqref{eq:nonadm-bootstrap-e}.  They also preserve, with room to spare,
$1/2<u,L<3/2$.  The usual open--closed continuation gives the estimates for
all $r\geq0$.  The variables $u,\tau,u',\tau'$ remain in a compact region on
every finite interval, so the radial ODE continuation criterion rules out
finite-radius blow-up.  In particular $u_s$ remains positive globally.

Equations \eqref{eq:nonadm-tau-improve}--\eqref{eq:nonadm-tau-L1} show that
$\tau(r)\to\tau_\infty>0$.  Put
\begin{equation*}
 \ell_s=(\tau_\infty/K)^{1/q},\qquad
 L_s=(\tau/K)^{1/q}.
\end{equation*}
Then
\begin{align*}
 |L_s(r)-\ell_s|\leq C\int_r^\infty t^{2-n}\,dt=O(r^{3-n}),
\end{align*}
which proves \eqref{eq:nonadm-L-tail}.  Moreover,
$k(r)^2-k_s^2=O(r^{3-n})\in L^1(1,\infty)$ and
$r^df\in L^1(1,\infty)$.  With
\begin{equation*}
 g_s=r^df+\left(k_s^2-k(r)^2
 +\frac{(n-1)(n-3)}{4r^2}\right)E,
\end{equation*}
equation \eqref{eq:nonadm-E} is $E''+k_s^2E=g_s$ and $g_s\in L^1(1,\infty)$.
Lemma~\ref{lem:radial-Helmholtz-scattering}, equivalently its
variation-of-constants proof applied directly to this equation, shows that
the sine and cosine coefficients have finite limits and that the remaining
tail is bounded by $C\int_r^\infty|g_s(t)|\,dt=o(1)$.  Hence
\begin{equation*}
 E(r)=A_s\cos(k_sr)+B_s\sin(k_sr)+o(1),
\end{equation*}
and the corresponding formula for $E'$.  This is
\eqref{eq:nonadm-u-tail}.

It remains to prove that the scattering vector is nonzero for small $s\ne0$.
Let $\nu=(n-2)/2$ and, for $\lambda>0$, define the regular radial Helmholtz mode
\begin{equation*}
 \phi_\lambda(r)=2^\nu\Gamma(\nu+1)(\sqrt\lambda r)^{-\nu}
 J_\nu(\sqrt\lambda r).
\end{equation*}
It satisfies $\phi_\lambda(0)=1$, $\phi_\lambda'(0)=0$ and
$\phi_\lambda''(0)=-\lambda/n$.  Its $r^{-d}$ scattering vector is nonzero
and depends smoothly on $\lambda$ in compact subsets of $(0,\infty)$.
Indeed, the standard Bessel expansion gives
\begin{align*}
 \phi_\lambda(r)
 =c_\nu\lambda^{-(2\nu+1)/4}r^{-d}
 \cos\left(\sqrt\lambda r-\frac{\pi\nu}{2}-\frac\pi4\right)+O(r^{-d-1}),\qquad c_\nu\ne0.
\end{align*}
There is a direct quadratic comparison which avoids assuming differentiability
of scattering data on an infinite interval.  Since $\tau'$ is quadratic in
$s$, \eqref{eq:nonadm-tau-improve} gives, uniformly in $r$,
\begin{equation}
 |\tau(r)-\tau_\infty|\leq C\delta^2,
 \qquad
 |\tau(r)-\tau_\infty|\leq C\delta^2r^{3-n}\quad(r\geq1).
 \label{eq:nonadm-tau-tail-small}
\end{equation}
Consequently
\begin{align*}
 |k(r)^2-k_s^2|&\leq C\delta^2,\\
 |k(r)^2-k_s^2|&\leq C\delta^2r^{3-n}\quad(r\geq1).
\end{align*}
Also
\begin{equation*}
 e(0)=1-\left(1-\frac nK s\right)^{1/q}
 =\frac n{\cN p}s+O(s^2),
 \qquad
 \ell_s=1-\frac n{\cN p}s+O(s^2),
\end{equation*}
where $qK=\cN p$ and the second expansion follows from
$\tau_\infty=K-ns+O(s^2)$.

Set $\rho_s=e-e(0)\phi_{k_s^2}$.  Then
\begin{equation*}
 \rho_s''+\frac{n-1}{r}\rho_s'+k_s^2\rho_s=G_s,
 \qquad \rho_s(0)=\rho_s'(0)=0,
\end{equation*}
with
\begin{equation*}
 G_s=(k_s^2-k(r)^2)e+f.
\end{equation*}
The estimates already proved, together with
\eqref{eq:nonadm-tau-tail-small}, imply
\begin{equation*}
 \|G_s\|_{L^\infty(0,1)}
 +\int_1^\infty r^d|G_s(r)|\,dr\leq C\delta^2.
\end{equation*}
The regular ODE estimate on $[0,1]$, followed by the transformed energy
estimate used for \eqref{eq:nonadm-E-bound} and
Lemma~\ref{lem:radial-Helmholtz-scattering}, shows that the scattering vector
of $\rho_s$ has norm at most $C\delta^2$.  If
$\mathcal S(\lambda)\ne0$ denotes the scattering vector of
$\phi_\lambda$, then therefore
\begin{equation*}
 (A_s,B_s)=e(0)\mathcal S(k_s^2)+O(s^2)
 =\frac n{\cN p}s\,\mathcal S(\cN p)+O(s^2).
\end{equation*}
Since $\mathcal S(\cN p)\ne0$, this proves $(A_s,B_s)\ne(0,0)$ for all
sufficiently small $s\ne0$.

Finally, from $z=Ku^q-\tau=K\{(L+e)^q-L^q\}$ we obtain
\eqref{eq:nonadm-z-tail}.  The identity
\begin{equation*}
 -\Delta(u^p)=p u^{p-1}
 \left(z-(p-1)\frac{|\nabla u|^2}{u}\right)
\end{equation*}
and $|u'|^2=O(r^{-2d})=o(r^{-d})$ give
\eqref{eq:nonadm-w-tail}.
\end{proof}

\begin{proof}[Proof of Theorem~\ref{thm:nonadm-scattering}]
The normalized conclusion is Proposition
\ref{prop:nonadm-scattering-precise}.  If $u$ is a solution, then
$x\mapsto c u(c^{p-1}x)$ is also a solution for every $c>0$.  Applying this
scaling to the normalized family produces the branch near the constant
background $c$.
\end{proof}

\subsection{A periodic rigidity consequence}
Although the scattering solutions oscillate indefinitely, they are not
periodic.  In fact, we have the following stronger rigidity statement.
\begin{proposition}
There is no positive nonconstant radial solution of
\eqref{eq:main} whose radial profile is periodic.
\end{proposition}

\begin{proof}
For a radial solution,
\begin{equation*}
 -\left(r^{n-1}T_r\right)_r=\cN p(p-1)r^{n-1}u^{q-2}u_r^2
 \geq0.
\end{equation*}
Since smoothness at the origin gives
\(\lim_{r\downarrow0}r^{n-1}T_r(r)=0\), it follows that
\begin{equation*}
 T_r(r)\leq0
 \qquad (r>0).
\end{equation*}

Suppose that \(P>0\) is a period of \(u\).  Since
\begin{equation*}
 T(r)=-u_{rr}(r)-\frac{n-1}{r}u_r(r)-Ku(r)^q,
\end{equation*}
the periodicity of \(u\) and its derivatives gives, for every \(r>0\),
\begin{equation*}
 T(r+P)-T(r)=\frac{(n-1)P}{r(r+P)}u_r(r).
\end{equation*}
The left-hand side is nonpositive because \(T_r\leq0\); hence
\(u_r(r)\leq0\) for every \(r>0\).  Thus \(u\) is both nonincreasing and
periodic, and must therefore be constant.
\end{proof}
\medskip
\noindent\textbf{Declaration of AI-assisted technologies.}\enspace
The central ideas of this work were conceived by the author. OpenAI's ChatGPT
was used to assist with language editing, exposition, and checking calculations
and logical consistency. All mathematical arguments were independently verified
by the author, who takes full responsibility for the content of the paper.

\end{document}